\documentclass{article}
\usepackage{amsthm}
\usepackage{amsfonts}
\usepackage{amsmath}
\usepackage{latexsym}
\usepackage{mathtools}
\usepackage{color}
\theoremstyle{definition}
\newtheorem{thm}{Theorem}
\newtheorem*{thma}{Theorem~A}
\newtheorem*{thmb}{Theorem~B}
\newtheorem{lem}{Lemma}
\newtheorem{prop}{Proposition}
\newtheorem{dfn}{Definition}
\newtheorem{cor}{Corollary}

\renewcommand{\labelenumi}{\theenumi)}

\begin{document}
\title{Convergence to pushed fronts and the behavior of level sets in monostable reaction-diffusion equations}
\author{Ryo Kiyono}
\date{}
\maketitle

\begin{abstract}
We study the behavior of solutions of a monostable reaction-diffusion equation $u_t=\Delta_x u +u_{yy} +f(u)$ ($x \in \mathbb{R}^{n-1}$, $y \in \mathbb{R}$, $t>0$), with the unstable equilibrium point $0$ and the stable equilibrium point $1$. Under the condition that the corresponding one-dimensional equation has a pushed front $\Phi_{c^*}(z)$ with $\Phi_{c^*}(-\infty)=1$, $\Phi_{c^*}(\infty)=0$, we show that the solution $u(x,y,t)$ approaches $\Phi_{c^*}(y-\gamma(x,t))$ for some $\gamma(x,t)$ as $t \to \infty$, if initially $u(x,y,0)$ decays sufficiently fast as $y \to \infty$ and is bounded below by some positive constant near $y=-\infty$. It is also shown that $\gamma(x,t)$ is approximated by the mean curvature flow with a drift term.
\end{abstract}

\section{Introduction}
In this paper, we consider the reaction-diffusion equation
\begin{equation}
\left\{
\begin{aligned}\label{Rea-Diff}
&u_{t}=\Delta u+f(u), && x\in\mathbb{R}^{n-1}, \, y\in\mathbb{R}, \, t>0, \\
&u(x,y,0)=u_{0}(x,y), && x\in\mathbb{R}^{n-1}, \, y\in\mathbb{R}.
\end{aligned}
\right.
\end{equation} 
Here $\Delta=\partial^2/\partial x_1^2+\cdots +\partial^2/\partial x_{n-1}^2+\partial^2/\partial y^2$ and $n\ge2$. Throughout the paper, the initial data $u_{0}$ is assumed to be nonnegative, bounded and uniformly continuous, and the reaction term $f$ is assumed to be of class $C^{1}$.
We are interested in the asymptotic behavior of solutions with non-compactly supported initial data in the case where the reaction term $f$ is of monostable-type.

We begin with recalling results on the existence and stability of traveling wave solutions in one dimensional monostable reaction diffusion equations. 
We consider the reaction term $f$ satisfying
\begin{equation*}
\mbox{(F)} \
\left\{
\begin{aligned}
&f(0)=f(1)=0,\quad f^{\prime}(0)>0,\quad f^{\prime}(1)<0,
\\
&f(s)>0 \ (s\in(0,1)),\quad f(s)<0\ (s\in(-\infty,0)\cup(1,\infty)).
\end{aligned}
\right.
\end{equation*}
Then it is well-known that there exists the minimal speed $c^{*}\ge2\sqrt{f^{\prime}(0)}$ for traveling front solutions joining the equilibria $1$ and $0$.
More precisely, for any $c\ge c^{*}$, equation \eqref{Rea-Diff} has a traveling front solution written in the form $u(x,t)=\Phi_{c}(x-ct)$ for a profile function $\Phi_{c}$ satisfying
\begin{gather*}
\Phi_{c}^{\prime\prime}+c\Phi_{c}^{\prime}+f(\Phi_{c})=0, \\
\lim_{z\to-\infty}\Phi_{c}(z)=1, \quad \lim_{z\to\infty}\Phi_{c}(z)=0.
\end{gather*}
It is known that there exist $\alpha, \beta\ge0$ with $(\alpha,\beta) \neq (0,0)$ such that
\begin{gather}
\label{asym}
\Phi_{c}(s)=(\alpha +o(1))e^{\lambda_+(c)s} \quad \mbox{if} \quad c>c^{*}, \\
\label{asym1}
\Phi_{c}(s)=(\alpha s+\beta +o(1))e^{\lambda_-(c)s} \quad \mbox{if} \quad c=c^{*},
\end{gather}
where $\lambda_{+}(c)$ and $\lambda_{-}(c)$ are the largest root and the smallest root of the quadratic equation
\begin{equation}
\lambda^{2}+c\lambda+f^{\prime}(0)=0,
\end{equation}
respectively. 
The traveling front solution $u(x,t)=\Phi_{c}(x-ct)$ is called a pulled front if either $c=c^{*}=2\sqrt{f^{\prime}(0)}$ or $c>c^{*}$ holds, and is called a pushed front if $c=c^{*}>2\sqrt{f^{\prime}(0)}$.

Concerning asymptotic behavior of solutions,
many results are known for the one-dimensional problem
\begin{equation*}
\left\{
\begin{aligned}
&u_{t}=u_{yy}+f(u), && y\in\mathbb{R}, \, t>0, \\
&u(y,0)=u_{0}(y), && y\in\mathbb{R}.
\end{aligned}
\right.
\end{equation*} 
In the pioneer work \cite{KPP}, it is shown that if $f(u)=u(1-u)$ and 
\begin{equation*}
u_0(y)=\left\{
\begin{aligned}
&1 \quad (y<0), \\
&0 \quad (y \ge 0),
\end{aligned}
\right.
\end{equation*} 
then $u(z+\sigma(t),t)$ converges uniformly to $\Phi_{c^{*}}(z)$ as $t \to \infty$
for some function $\sigma(t)$ satisfying
\begin{equation*}
\sigma(t)=2t+o(t) \quad (t \to \infty).
\end{equation*}
The refined behavior of $\sigma(t)$ is revealed in \cite{MR705746,MR494541}. More precisely, it is shown that $\sigma(t)$ satisfies
\begin{equation*}
\sigma(t)=2t-\frac{3}{2}\ln t+z_* +o(1)  \quad (t \to \infty)
\end{equation*}
for some number $z_*$.
Similar results for more general reaction terms and initial functions are obtained in \cite{MR4412555,MR422875,MR803086,MR509494}.

For pushed fronts, Stokes \cite{MR682241} and Rothe \cite{MR639447} proved that, if the initial data $u_{0}(y)$ satisfies
\begin{equation*}
0\le u_0(y)\le1, \qquad
\liminf_{y\rightarrow-\infty}u_0(y)>0, \qquad
u_0(y)\le Ke^{\lambda y},
\end{equation*}
for some constants $K>0$ and $\lambda<\lambda_+(c_*)$, then 
\begin{equation}
u(z+c^*t,t)\rightarrow \Phi_{c^*}(z+\xi) \quad (t\rightarrow\infty)
\end{equation}
for some constant $\xi$. 
In contrast to pulled fronts with the minimal speed,
the logarithmic correction term does not appear for pushed fronts.
This is analogous to the result in the bistable case \cite{MR442480}.

In higher dimensional cases, the pioneering work has done by Aronson and Weinberger \cite{MR511740}. 
They prove that if the initial data has compact support and satisfies $0 \le u_0 \le 1$, $u_0 \not\equiv0$, then
\begin{equation*}
\lim_{t \to \infty} \sup_{|x|+|y| \ge (c^*+\delta)t} |u(x,y,t)|=0,
\qquad
\lim_{t \to \infty} \sup_{|x|+|y| \le (c^*-\delta)t} |u(x,y,t)-1|=0,
\end{equation*}
for any $\delta>0$. Since then, the large-time behavior of solutions with compactly supported initial data has been extensively studied \cite{MR3333711,MR670523, MR4026189,MR801583}. 

Our interest is the asymptotic behavior of solutions with non-compactly supported initial data. In contrast to the case where the initial data has compact support, less is known about the behavior of such solutions.
To observe what can occur, we recall results for bistable reaction diffusion equations established by Matano and Nara \cite{MR2837694}.
Under some mild assumptions on the initial data,
they showed  the convergence of a solution to $\Psi(y-\gamma(x,t))$ for some function $\gamma(x,t)$, where $\Psi$ denotes a one-dimensional traveling wave solution. Moreover, they also found that $\gamma(x,t)$ is approximated by the mean curvature flow with a drift term.
To be more precise, the following result is proved.

\begin{thma}[\cite{MR2837694}]
Suppose that $f$ satisfies
\begin{gather*}
f(0)=f(1)=0,\quad f^{\prime}(0)<0,\quad f^{\prime}(1)<0, \quad
f(s) \left\{ 
\begin{aligned}
&>0 &&\mbox{if } s \in (-\infty,0), \\
&<0 &&\mbox{if } s \in (1,\infty)
\end{aligned}
\right.
\end{gather*}
and that there exist $c\in\mathbf{R}$ and $\Psi \in C^2 (\mathbf{R})$ satisfying
\begin{gather*}
\Psi^{\prime\prime}+c\Psi^{\prime}+f(\Psi)=0, \quad
\lim_{z\to-\infty}\Psi(z)=1, \quad \lim_{z\to\infty}\Psi(z)=0.
\end{gather*}
Put 
\begin{gather*}
s_+:=\inf \{ s_0 \in (0,1); \, f>0 \mbox{ on } (s_0,1)\},\\
s_-:=\sup \{ s_0 \in (0,1); \, f<0 \mbox{ on } (0,s_0)\},
\end{gather*}
and suppose that $u_0$ satisfies
\begin{equation}
\liminf_{y\rightarrow-\infty}\inf_{x\in\mathbf{R}^{n-1}}u_0 (x,y)>s_+,\quad \limsup_{y\rightarrow\infty}\sup_{x\in\mathbf{R}^{n-1}}u_0 (x,y)<s_-.
\end{equation}
Let $u$ be a solution of \eqref{Rea-Diff}.
Then there exists a smooth function  $\gamma=\gamma(x,t)$ with the following properties.
\begin{enumerate}
\item[(i)] 
There exists $T>0$ such that 
\begin{equation*}
\{ (x,y,t) \in \mathbf{R}^{n-1} \times \mathbf{R} \times [T,\infty); \, u(x,y,t)=\Psi(0)\}=\{y=\gamma(x,t)\}.
\end{equation*}
Moreover, it holds that
\begin{equation*}
\lim_{t\rightarrow\infty}\sup_{(x,y)\in\mathbf{R}^n }|u(x,y,t)-\Psi(y-\gamma(x,t))|=0.
\end{equation*}
\item[(ii)]
For any $\varepsilon>0$, there exists $\tau_{\varepsilon}\in[T,\infty)$ such that the solution $U(x,t)$ of the problem
\begin{equation*}
\left\{
\begin{aligned}
&\frac{U_t}{\sqrt{1+|\nabla_x U|^2}}=\mathrm{div}\left(\frac{\nabla_x U}{\sqrt{1+|\nabla_x U|}}\right)+c,&&x\in\mathbf{R}^{n-1},t>0,\\
&U(x,0)=\gamma(x,\tau_{\varepsilon}),&&x\in\mathbf{R}^{n-1},
\end{aligned}
\right.
\end{equation*}
satisfies
\begin{equation*}
\sup_{x\in\mathbf{R}^{n-1},t\ge\tau_{\varepsilon}}|\gamma(x,t)-U(x,t-\tau_{\varepsilon})|\le\varepsilon.
\end{equation*}
\end{enumerate}
\end{thma}

In the case where $f$ is of monostable type, a similar result is obtained by Wang \cite{MR3746497}.
\begin{thmb}[\cite{MR3746497}]
In addition to (F), assume that 
\begin{equation}\label{assmpf}
0<f'(0)u-f(u)<Mu^{1+\alpha}
\quad \mbox{if } \ u \in (0,1)
\end{equation}
for some constants $M>0$ and $\alpha \in (0,1]$.
Then, there exists $\delta_0>0$ such that if the initial data $u_0$ satisfies
\begin{gather}
\liminf_{y \to -\infty} \inf_{x \in \mathbb{R}^{n-1}} u_0(x,y)>1-\delta_0, \label{lbidni} \\
0<\liminf_{y\to\infty}\inf_{x\in\mathbf{R}^{n-1}} \frac{u_0(x,y)}{\Phi_c(y)}, \quad
\limsup_{y\to\infty}\sup_{x\in\mathbf{R}^{n-1}}\frac{u_0(x,y)}{\Phi_c(y)}<\infty 
\ \mbox{ for some } \ c \ge c_*, \label{init2}
\end{gather}
the following are true.
\begin{enumerate}
\item[(i)]
The assertions of Theorem~A (i) with $\Psi$ replaced by $\Phi_c$ hold for some smooth function $\gamma=\gamma(x,t)$.
\item[(ii)]
For any $\varepsilon>0$, there exists $T_{\varepsilon}>0$ such that the inequalities
\begin{equation*}
v^-(x,t)-\varepsilon\le\gamma(x,t)\le v^+(x,t)+\varepsilon,\quad t\ge T_{\varepsilon},
\end{equation*}
hold for the solutions $v^-$ and $v^+$ of the initial value problems
\begin{equation*}
\left\{
\begin{aligned}
v^-_t&=\Delta_x v^--k|\nabla_x v^-|^2+c, && x\in\mathbf{R}^{n-1}, \, t>0,\\
v^-(x,0)&=\gamma(x,T_{\varepsilon}),         && x\in\mathbf{R}^{n-1},
\end{aligned}
\right.
\end{equation*}
\begin{equation*}
\left\{
\begin{aligned}
v^+_t&=\Delta_x v^++k|\nabla_x v^+|^2+c, && x\in\mathbf{R}^{n-1}, \, t>0,\\
v^+(x,0)&=\gamma(x,T_{\varepsilon}),         && x\in\mathbf{R}^{n-1},
\end{aligned}
\right.
\end{equation*}
where $k:=\sup_{z\in\mathbf{R}}|\Phi^{\prime\prime}_c(z)|/|\Phi^{\prime}_c(z)|$.
\end{enumerate}
\end{thmb}

It is well-known that condition \eqref{assmpf} implies $c_*=2\sqrt{f'(0)}$, which means that there is no pushed front under condition \eqref{assmpf}. The purpose of this paper is to reveal the behavior of $\gamma(x,t)$ when $\Phi_c(x-ct)$ is a pushed front. First, we verify that the same assertion as in Theorem~B (i) holds 
when condition \eqref{assmpf} is dropped and conditions \eqref{lbidni} and \eqref{init2} are replaced with
\begin{equation}
\label{init1}
\liminf_{y\to-\infty}\inf_{x\in\mathbf{R}^{n-1}}u_{0}(x,y)>0 \\
\end{equation}
and
\begin{equation}
\label{init3}
c=c_*>2\sqrt{f'(0)}, \qquad 
\limsup_{y\to\infty}\sup_{x\in\mathbf{R}^{n-1}} u_{0}(x,y)e^{-\lambda_1 y}<\infty \ \mbox{ for some } \ \lambda_1<\lambda_+,
\end{equation}
respectively. We then prove that the behavior of $\gamma(x,t)$ is governed by the mean curvature flow with a drift term, as in the case where $f$ is of bistable type. More precisely, our main result is stated as follows.


\begin{thm}\label{thm1}
Assume (F), \eqref{init1} and \eqref{init3} and let $u$ be a solution of \eqref{Rea-Diff}. Then there exists a smooth function  $\gamma=\gamma(x,t)$ with the following properties.
\begin{enumerate}
\item[(i)] 
There exists $T>0$ such that 
\begin{equation*}
\{ (x,y,t) \in \mathbf{R}^{n-1} \times \mathbf{R} \times [T,\infty); \, u(x,y,t)=\Phi_c(0)\}=\{y=\gamma(x,t)\}.
\end{equation*}
Moreover, it holds that
\begin{equation*}
\lim_{t\rightarrow\infty}\sup_{(x,y)\in\mathbf{R}^n }|u(x,y,t)-\Phi_c(y-\gamma(x,t))|=0.
\end{equation*}
\item[(ii)]
For any $\varepsilon>0$, there exists $\tau_{\varepsilon}\in[T,\infty)$ such that the solution $U(x,t)$ of the problem
\begin{equation}\label{MCF}
\left\{
\begin{aligned}
&\frac{U_t}{\sqrt{1+|\nabla_x U|^2}}=\mathrm{div}\left(\frac{\nabla_x U}{\sqrt{1+|\nabla_x U|}}\right)+c_*,&&x\in\mathbf{R}^{n-1},t>0,\\
&U(x,0)=\gamma(x,\tau_{\varepsilon}),&&x\in\mathbf{R}^{n-1},
\end{aligned}
\right.
\end{equation}
satisfies
\begin{equation*}
\sup_{x\in\mathbf{R}^{n-1},t\ge\tau_{\varepsilon}}|\gamma(x,t)-U(x,t-\tau_{\varepsilon})|\le\varepsilon.
\end{equation*}
\end{enumerate}
Furthermore, the assertion (i) still holds if \eqref{init2} is assumed instead of \eqref{init3}.
\end{thm}

We prove Theorem~\ref{thm1} by constructing appropriate comparison functions. They are given in the form
\begin{equation*}
u^{\pm}(x,y,t)\coloneqq\Phi_{c^{*}}\left(\frac{y-V(x,t)}{\sqrt{1+|\nabla_{x}V|^{2}}}\mp q(t)\right)\pm p(t)\chi \left( e^{\lambda (y-c_*t)}\right),
\end{equation*}
where $V$ is a solution of the equation
\begin{equation*}
V_{t}=\Delta_{x}V+\frac{c^{*}}{2}|\nabla_{x} V|^2+c_*,
\end{equation*}
$\lambda$ is a number with $\lambda<\lambda_1$
and $\chi$ is a smooth function safisfying $\chi(s)=s$ ($s \le 1/2$) and $\chi(s)=1$ ($s \ge 1$).
We will show that $u^+$ (resp. $u^-$) becomes a supersolution (resp. a subsolution) for the problem \eqref{Rea-Diff} if $p(t)$, $q(t)$ and the initial data for $V(x,t)$ are chosen appropriately. Theorem~\ref{thm1} (ii) is then proved by using these comparison functions and applying the fact that the solution $U$ of \eqref{MCF} is approximated by $V$. 

This paper is organized as follows. In Section 2, we recall results on the approximation of the mean curvature flow obtained in \cite{MR2837694}. Section 3 establishes upper and lower bounds of solutions at large time. In Section 4, we define $\omega$-limit points and provide their characterization. Section 5 establishes smoothness of level sets of solutions. Section 6 is devoted to the construction of comprison functions, which are used to prove Theorem \ref{thm1} (ii). In Section 7, we prove Theorem \ref{thm1}.

\renewcommand{\theenumi}{\roman{enumi}}
\renewcommand{\labelenumi}{(\theenumi)}

\section{Approximation of mean curvature flow}
 In this section, we present the following lemma, which shows that the mean curvature flow can be approximated by a semilinear equation under certain circumstances. The lemmas in this section are same as \cite{MR2837694}, so we omit their proof.
\begin{lem}[Approximation of the mean curvature flow, \cite{MR2837694}]\label{lem1}
Let $U(x,t;\phi)$ and $V(x,t;\phi)$ denote the solutions of equations
\begin{equation*}
\frac{U_{t}}{\sqrt{1+|\nabla_{x}U|^{2}}}=\mbox{div}\left(\frac{\nabla_{x}U}{\sqrt{1+|\nabla_{x}U|^{2}}}\right)+c,\quad x\in\mathbb{R}^{m},t>0,
\end{equation*}
\begin{equation*}
V_{t}=\Delta_{x}V+\frac{c}{2}|\nabla_{x}V|^{2}+c,\quad x\in\mathbb{R}^{m},t>0,
\end{equation*}
under the initial conditions $U(\cdot,0)=V(\cdot,0)=\phi\in W^{2,\infty}(\mathbb{R}^{m})$. Then, for any constant $\varepsilon>0$, there exists a constant $\delta>0$ such that if $\|\nabla_{x}\phi\|_{W^{1,\infty}}\le\delta$, it holds that 
\begin{equation*}
\sup_{x\in\mathbb{R}^{m}}|U(x,t;\phi)-V(x,t;\phi)|\le\varepsilon\quad for\; all\quad t\ge0.
\end{equation*}
\end{lem}

\begin{lem}[\cite{MR2837694}]\label{lem2}
Let $V(x,t)$ be a solution to the problem
\begin{equation*}
\left\{
\begin{aligned}
&V_{t}=\Delta_{x}V+\frac{c}{2}|\nabla_{x}V|^{2}+c,\quad x\in\mathbb{R}^{m},t>0,\\
&V(x,0)=V_{0}(x),\hspace{54pt}x\in\mathbb{R}^{m}.
\end{aligned}
\right.
\end{equation*}
Then the following estimates hold:
\begin{align*}
\sup_{x\in\mathbb{R}^{m}}|V_{x_{i}}(x,t)|&\le\min\{C_{0}t^{-\frac{1}{2}},C_{1}\},\\
\sup_{x\in\mathbb{R}^{m}}|V_{x_{i}x_{j}}|&\le\min\{C_{0}t^{-1},C_{2}\},\\
\sup_{x\in\mathbb{R}^{m}}|V_{x_{i}x_{j}x_{k}}|&\le C_{3}(1+t)^{-\frac{3}{2}},\\
\sup_{x\in\mathbb{R}^{m}}|V_{x_{i}t}|&\le C_{4}(1+t)^{-\frac{3}{2}},
\end{align*}
for each $1\le i,j,k\le m$, where $C_{0},C_{1},C_{2},C_{3}$ and $C_{4}$ are positive constants such that
\begin{enumerate}
\item $C_{0}$ depends only on $c$ and $\|V_{0}\|_{L^{\infty}}$,
\item $C_{1}$ depends only on $c, \|V_{0}\|_{L^{\infty}}$ and satisfies
\begin{equation*}
C_{1}\to0\quad as\quad \|\nabla_{x}V_{0}\|_{L^{\infty}}\to0,
\end{equation*}
\item $C_{2}$ depends only on $c,\|V_{0}\|$ and $\|\nabla_{x}V_{0}\|_{L^{\infty}}$ and satisfies
\begin{equation*}
C_{2}\to0\quad as\|\nabla_{x}V_{0}\|_{W^{1,\infty}}\to0,
\end{equation*}
\item $C_{3}$ and $C_{4}$ depend only on $c$ and $\|V_{0}\|_{W^{3,\infty}}$.
\end{enumerate}
\end{lem}

\section{Upper and lower bounds of solutions}

Throughout this section, we always assume that $u_0$ satisfies the conditions in Theorem~\ref{thm1}.
Without loss of generality, we may assume that 
\begin{equation}\label{condf}
f(u)=f^{\prime}(0)u \quad (u<0),
\end{equation}
because the nonnegativity of $u_{0}$ and the maximum principle show that the solution $u$ of \eqref{Rea-Diff} is nonnegative. 

We introduce a moving frame. Write
\begin{equation*}
\label{Mov}
z=y-ct.
\end{equation*}
Then \eqref{Rea-Diff} is transformed into
\begin{equation}\label{MRea-Diff}
\left\{
\begin{aligned}
&u_{t}=\Delta u+cu_{z}+f(u),\quad x\in\mathbb{R}^{n-1}, \, z\in\mathbb{R}, \, t>0, \\
&u(x,y,0)=u_{0}(x,y), \hspace{20pt}  x\in\mathbb{R}^{n-1}, \, z\in\mathbb{R},
\end{aligned}
\right.
\end{equation} 
where $\Delta=\partial^{2}/\partial x_{1}^{2}+\cdots+\partial^{2}/\partial x_{n-1}^{2}+\partial^{2}/\partial z^{2}$. 

In this section, we establish the following estimates.
\begin{prop}\label{lem3}
Let $u(x,z,t)$ be a solution of \eqref{MRea-Diff}. 
If the initial data $u_{0}$ is satisfies \eqref{init1} and \eqref{init3},
then there exists constants $z_{1},z_{2}\in\mathbb{R}$ such that
\begin{equation}
\label{upper}
\limsup_{t\to\infty}\sup_{x\in\mathbb{R}^{n-1}}u(x,z,t)\le\Phi_{c}(z-z_{0})\;\mbox{uniformly in }z\in\mathbb{R}
\end{equation}
\begin{equation}
\label{lower}
\liminf_{t\to\infty}\inf_{x\in\mathbb{R}^{n-1}}u(x,z,t)\ge\Phi_{c}(z-z_{1})\;\mbox{uniformly in }z\in\mathbb{R}
\end{equation}
The same inequalities hold if \eqref{init2} is assumed instead of \eqref{init3}.
\end{prop}
We split the proof of this proposition into the case 1 and case 2:
\begin{enumerate}
\item[1:]\eqref{init1} and \eqref{init3}
\item[2:]\eqref{init1} and \eqref{init2}
\end{enumerate}
First, we show upper and lower estimates for case 1.
For this, we recall the lemmas in \cite{MR639447}. We take any $\lambda<\lambda_{1}<\lambda_{+}$ and define $\psi(s)$ as
\begin{equation*}
\psi(s)=\chi(e^{\lambda_{1}s})
\end{equation*}
\begin{equation*}
\chi(s)\coloneqq\left\{
\begin{aligned}
1\;(s\ge1)\\
s\;(s\le\frac{1}{2})
\end{aligned}
\right.
\end{equation*}
, where $0\le\chi(s)\le1$ for $s\in(\frac{1}{2},1)$. For these, the following lemmas hold
\begin{lem}[\cite{MR639447}]\label{add_B}
There exists $p\in(0,1)$ such that, for any $q_{0}\in(0,p],z_{1},z_{2}\in\mathbf{R}$, there exists $\beta>0,C>0$ such that
\begin{equation*}
w^{+}(z,t)\coloneqq \Phi_{c^{*}}(z-z_{1}- C(1-e^{-\beta t}))+q_{0}e^{-\beta t}\psi(z-z_{2})
\end{equation*}
satisfies $L[w^+]\ge0$.
\end{lem}

\begin{lem}[\cite{MR639447}]\label{B}
For any $q_{0}\in(0,1)$, there exists $\beta>0,C>0$ such that, for any $z_{1},z_{2}\in\mathbf{R}$, 
\begin{equation*}
w^{-}(z,t)\coloneqq \Phi_{c^{*}}(z-z_{1}+ C(1-e^{-\beta t}))-q_{0}e^{-\beta t}\psi(z-z_{2})
\end{equation*}
satisfies $L[w^-]\le0$.
\end{lem}

Then, we show Proposition \ref{lem3} for case 1.

\begin{proof}[Proof of Proposition \ref{lem3} for case 1]
First, we take  $q_{0}\in(0,1)$ such that $\displaystyle \liminf_{z\to-\infty}\inf_{x\in\mathbf{R}}u_{0}(x,z)>1-q_{0}$. Then, there exists $M>0$ such that, for $z\le -M$, 
\begin{equation*}
u_{0}(x,z)>1-q_{0}.
\end{equation*}
Then, we take $z_{2}=-M$ and $z_{1}$ such that, for any $z>-M$,
\begin{equation*}
w^{-}(z,0)<0.
\end{equation*} 
Moreover, we obtain that, for any $z\le -M$, 
\begin{equation*}
w^{-}(z,0)\le 1-q_{0}<u_{0}(x,z).
\end{equation*}
So, we obtain a lower bound.

Next, we prove upper estimate. If $a>0$ is sufficiently large, $w(z,t)\coloneqq e^{\lambda(z-z_{0})+at}$ becomes supersolution, Thus, by assumption about an initial data, if $z_{0}$ is sufficiently large, $u(x,z,t)\le e^{\lambda(z-z_{0})+at}$ and, for sufficiently large $T>0$, 
\begin{equation*}
u(x,z,T)\le1+\frac{q_{0}}{2},
\end{equation*}
and we take $M>0, z_{2}\ge M$ such that, for any $z\ge M$,
\begin{equation*}
w^{+}(z,0)\ge u(x,z,T),
\end{equation*}
and $z_{1}$ such that, for any $z<M$,
\begin{equation*}
\Phi_{c^{*}}(z-z_{1})\ge 1-\frac{q_{0}}{2}.
\end{equation*}
Then, for any $z<M$, 
\begin{equation*}
w^{+}(z,0)=\Phi_{c^{*}}(z-z_{1})+q_{0}\ge 1-\frac{q_{0}}{2}\ge u(x,z,t),
\end{equation*}
and, for any $z\ge M$, 
\begin{equation*}
w^{+}(z,0)\ge u(x,z,T)
\end{equation*}
Therefore, we obtain an upper bound.
\end{proof}

Second, we show upper and lower estimates for case 2. For the proof, we first recall the following lemma.

\begin{lem}[\cite{MR3746497}]\label{lem4}
There exists $\varepsilon_{0}\in(0,1), \beta>0$ such that, for any $\varepsilon\in(0,\varepsilon_{0}]$, there exists $\sigma>0$ such that the functions given by
\begin{equation}
\begin{aligned}
&u^{+}(z,t)\coloneqq (1+\varepsilon e^{-\beta t})\Phi_{c}(z-\sigma\varepsilon(1-e^{-\beta t})), \\
&u^{-}(z,t)\coloneqq (1-\varepsilon e^{-\beta t})\Phi_{c}(z+\sigma\varepsilon(1-e^{-\beta t}))
\end{aligned}
\end{equation}
satisfy
\begin{equation}
\begin{aligned}
&\mathrm{L}[u^{+}]\coloneqq u_{t}^{+}-u_{zz}^{+}-cu_{z}^{+}-f(u^{+})\ge0, \\
&\mathrm{L}[u^{-}]\coloneqq u_{t}^{-}-u_{zz}^{-}-cu_{z}^{-}-f(u^{-})\le0
\end{aligned}
\end{equation}
\end{lem}

In order to obtain the upper and lower bounds of the solution $u(x,z,t)$ at $z=\infty$, we show the following lemma.
\begin{lem}\label{lem5}
For sufficiently large $a>0$ and any $z_{0}\in\mathbb{R}$, 
the functions
\begin{equation*}
w^{\pm}(z,t)=(1\pm e^{-(z-at)})\Phi_{c}(z-z_{0})
\end{equation*}
satisfy
\begin{equation*}
\mathrm{L}[w^{+}] \ge 0, \qquad \mathrm{L}[w^{-}] \le 0.
\end{equation*}
\end{lem}

\begin{proof}
By a direct computation, we have
\begin{align*}
\mathrm{L}[w^{\pm}]=\pm(a+1+c+2\Phi_{c}^{\prime}/\Phi_{c})\Phi_{c} e^{-(z-at)}+(1\pm e^{-(z-at)})f(\Phi_{c})-f((1\pm e^{-(z-at)})\Phi_{c})
\end{align*}
Then, we take $k>0,a>0$ such that $|\Phi_{c}^{\prime}|/\Phi_{c}\le k$ and $a\ge2k-1-c+\|f^{\prime}\|_{L^{\infty}[0,1]}$.
First, we prove that $u^{+}$ is a supersolution. If $w^{+}\ge1$, 
\begin{align*}
\mathrm{L}[w^{+}]&\ge(a+1+c-2k)\Phi_{c}e^{-(z-at)}\\
&\ge0.
\end{align*}
If $w^{+}\le1$,
\begin{align*}
\mathrm{L}[w^{+}]&\ge(a+1+c-2k-\|f^{\prime}\|_{L^{\infty}[0,1]})\Phi_{c}e^{-(z-at)}\\
&\ge0
\end{align*}
Second, we prove that $w^{-}$ is a subsolution. In view of \eqref{condf}, we have
\begin{align*}
\mathrm{L}[w^{-}]&\le-(a+1+c-2k-\|f^{\prime}\|_{L^{\infty}[0,1]})\Phi_{c}e^{-(z-at)}\\
&\le0
\end{align*}
Thus the proof is complete.
\end{proof}

\noindent Next, we prove Proposition \ref{lem3}.

\begin{proof}[Proof of Proposition \ref{lem3} for case 2]
By Lemma \ref{lem4}, it is sufficient to prove that, for some $T>0$, there exists $z_{1},z_{2}\in\mathbb{R}$ such that
\begin{equation*}
(1-\varepsilon_{0})\Phi_{c}(z-z_{1})\le u(x,z,T)\le(1+\varepsilon_{0})\Phi_{c}(z-z_{2}).
\end{equation*}
First, we prove the upper bound. For some $T>0$,
\begin{equation*}
u(x,z,T)\le1+\frac{\varepsilon_{0}}{2}\;((x,z)\in\mathbf{R}^{n})
\end{equation*}
and, by assumption of an initial data and (\ref{asym}), there exists $M>0$ and $z_{1}\in\mathbb{R}$ such that
\begin{equation*}
u_{0}(x,z)\le\Phi_{c}(z-z_{1})\;(z\ge M).
\end{equation*}
Then, if we take $z_{2}\in\mathbb{R}$ such that $\|u_{0}\|_{L^{\infty}}\le(1+e^{-(M-z_{2})})\Phi_{c}(M-z_{1})$, 
\begin{equation*}
u_{0}(x,z)\le\Phi_{c}(z-z_{1})\le(1+e^{-(z-z_{2})})\Phi_{c}(z-z_{1})\;(z\ge M)
\end{equation*}
\begin{equation*}
u_{0}(x,z)\le(1+e^{-(M-z_{2})})\Phi_{c}(M-z_{1})\le(1+e^{-(z-z_{2})})\Phi_{c}(z-z_{1})\;(z\le M)
\end{equation*}
Therefore, by Lemma 3, we have
\begin{equation*}
u_{0}(x,z)\le(1+e^{-(z-z_{2}-at)})\Phi_{c}(z-z_{1}).
\end{equation*}
We take $T>0$ such that
\begin{equation*}
u(x,z,T)\le 1+\frac{\varepsilon_{0}}{2}.
\end{equation*}
Then, $u(x,z,T)\le(1+e^{-(z-z_{2}-aT)})\Phi_{c}(z-z_{1})$ and by (\ref{asym}) there exists $M_{1}>0$ and $z_{3}\in\mathbf{R}$ such that
\begin{equation*}
u(x,z,T)\le\Phi_{c}(z-z_{3})\;(z\ge M_{1})
\end{equation*}
Therefore, if we take $z_{4}\ge z_{3}$ such that $\Phi_{c}(M_{1}-z_{4})\ge\frac{1+\frac{\varepsilon_{0}}{2}}{1+\varepsilon_{0}}$, we have that, for any $z\ge M_{1}$, 
\begin{align*}
u(x,z,T)&\le\Phi_{c}(z-z_{3})\\
&\le\Phi_{c}(z-z_{4})\le(1+\varepsilon_{0})\Phi_{c}(z-z_{4})
\end{align*}
For any $z\le M_{1}$,
\begin{align*}
u(x,z,T)&\le1+\frac{\varepsilon_{0}}{2}\\
&\le(1+\varepsilon_{0})\Phi_{c}(M_{1}-z_{4})\\
&\le(1+\varepsilon_{0})\Phi_{c}(z-z_{4}).
\end{align*}
From above, we obtain upper bound.

Second, we prove the lower bound. We define Heviside function $H(x)$ as
\begin{equation*}
H(x)\coloneqq\left\{
\begin{aligned}
&0\ (x<0)\\
&1\ (x\ge0).
\end{aligned}
\right.
\end{equation*}
Then, by assumption, there exists $M>0,\delta>0$ such that
\begin{equation}
u_{0}(x,z)\ge\delta>0\ (z<M)
\end{equation}
Therefore, we take $w(z,t)$ such that
\begin{equation*}
\left\{
\begin{aligned}
&w_{t}=w_{zz}+cw_{z}+f(w)\\
&w(z,0)=\delta(1-H(z-M)).
\end{aligned}
\right.
\end{equation*}
Then, since $w$ is nonincreasing and there is hair-Trigger effect in the original coordinate, there exists $T>0$ such that $w(-cT,T)\ge1-\frac{\varepsilon_{0}}{2}$.
Therefore, by $u(x,z,T)\ge w(z,T)$, if $z\le-cT$
\begin{equation*}
u(x,z,T)\ge w(z,T)\ge w(-cT,T)\ge1-\frac{\varepsilon_{0}}{2},
\end{equation*}
and, by assumption and (\ref{asym}), there exists $M>0,z_{1}\in\mathbf{R}$ such that
\begin{equation*}
u_{0}(x,z)\ge\Phi_{c}(z-z_{1})\;(z\ge M)
\end{equation*}
So, when we take $z_{2}\in\mathbf{R}$ such that $e^{-(M-z_{2})}\ge1$, for any $z\le M$,
\begin{equation*}
u_{0}(x,z)\ge0\ge(1-e^{-(z-z_{2})})\Phi_{c}(z-z_{1})
\end{equation*}
If $z\ge M$, 
\begin{equation*}
u_{0}(x,z)\ge\Phi_{c}(z-z_{1})\ge(1-e^{-(z-z_{2})})\Phi_{c}(z-z_{1}).
\end{equation*}
Therefore, by Lemma \ref{lem5}, we have
\begin{equation*}
u(x,z,t)\ge(1-e^{-(z-z_{2}-at)})\Phi_{c}(z-z_{1}).
\end{equation*}
By this and (\ref{asym}), there exists $M_{1},z_{3}$ such that, if $M_{1}\le z$, we have
\begin{equation*}
u(x,z,T)\ge\Phi_{c}(z-z_{3}).
\end{equation*}
From the fact that $w(z,t)>0\;(t>0)$, for any $-cT\le z \le M_{1}$, $u(x,z,T)\ge w(z,T)\ge w(M_{1},T)>0$. Then, we can take $z_{4}\le z_{3}$ such that 
\begin{equation*}
u(x,z,T)\ge\Phi_{c}(z-z_{4})\quad(-cT\le z \le M_{1}).
\end{equation*}
From above, for any $z>M_{1}$, 
\begin{align*}
u(x,z,T)&\ge\Phi_{c}(z-z_{3})\\
&\ge\Phi_{c}(z-z_{4}).
\end{align*}
For any $-cT\le z\le M_{1}$,
\begin{equation*}
u(x,z,T)\ge\Phi_{c}(z-z_{4})
\end{equation*}
For any $z\le -cT$
\begin{align*}
u(x,z,T)&\ge1-\frac{\varepsilon_{0}}{2}\\
&\ge(1-\varepsilon_{0})\Phi_{c}(z-z_{3})\\
&\ge(1-\varepsilon_{0})\Phi_{c}(z-z_{4})
\end{align*}
By these, we obtain lower bound.
\end{proof}


\section{$\omega$-limit points}
 In this section, we first introduce the notion of $\omega$-limit points of the solution $u(x,z,t)$ of (\ref{MRea-Diff}). Then, we show that any $\omega$-limit point is a planar wave under assumptions. This definition is the same as that of \cite{MR2837694}.
 
\begin{dfn}[\cite{MR2837694}]
A function $w(x,z,t)$ defined on $\mathbb{R}^{n-1}\times\mathbb{R}\times\mathbb{R}$ is called an $\omega$-limit point of the solution $u(x,z,t)$ of (\ref{MRea-Diff}) if there exists a sequence $\{(x_{i},t_{i})\}$ such that $0<t_{1}<t_{2}<\cdots\to\infty$ and that
\begin{equation*}
u(x+x_{i},z,t+t_{i})\to w(x,z,t)\quad\mbox{as}\;i\to\infty\;\mbox{in}\; C_{loc}^{2,1}(\mathbb{R}^{n}\times\mathbb{R}).
\end{equation*}
\end{dfn}

  Berestycki and Hamel \cite{MR2373726} obtained the following result that states that any entire solution of monostable reaction-diffusion equation lying between two planar waves is itself a planar waves.
 
\begin{lem}[\cite{MR2373726}]\label{lem7}
Let $u(x,z,t)$ be a function that is defined on $\mathbb{R}^{n-1}\times\mathbb{R}\times\mathbb{R}$ and satisfies
\begin{equation*}
u_{t}=\Delta u+cu_{z}+f(u),\quad(x,z)\in\mathbb{R}^{n},t\in\mathbb{R}.
\end{equation*}
Assume further that there exists three constants $z_{*},z^{*},c\in\mathbb{R}$ such that 
\begin{equation*}
\Phi_{c}(z-z_{*})\le u(x,z,t)\le \Phi_{c}(z-z^{*}),\quad(x,z)\in\mathbb{R}^{n},t\in\mathbb{R}.
\end{equation*}
Then there exists a constant $z_{0}\in[z_{*},z^{*}]$ such that
\begin{equation*}
u(x,z,t)=\Phi(z-z_{0}),\quad(x,z)\in\mathbb{R}^{n},t\in\mathbb{R}.
\end{equation*}
\end{lem}

The rest of  lemmas and corollaries of this section and next section can be proved in a similar way of \cite{MR2837694}, hence we only state the results and omit their proofs.

\begin{cor}[Characterization of $\omega$-limit points]
Let $u(x,z,t)$ be a solution of (\ref{MRea-Diff}).Then any $\omega$-limit point $w(x,z,t)$ of u is a planar wave, that is, there exists a constant $z_{0}\in\mathbb{R}$ such that
\begin{equation*}
w(x,z,t)=\Phi_{c}(z-z_{0}),\quad(x,z)\in\mathbb{R}^{n},t\in\mathbb{R}.
\end{equation*}
\end{cor}

From this, we can prove the following result.

\begin{lem}[Monotonicity in $z$]\label{lem8}
Let $u(x,z,t)$ be a solution of (\ref{MRea-Diff}). Then for any constant $R>0$, there exists a constant $T>0$ such that 
\begin{equation*}
\inf_{x\in\mathbb{R}^{n-1},|z|\le R, t\ge T}-u_{z}(x,z,t)>0
\end{equation*}
\end{lem}

\begin{cor}[Monotonicity in $z$ around the level-set]\label{lem9}
Let $u(x,z,t)$ be (\ref{MRea-Diff}). Then there exists a constant $T>0$ such that
\begin{equation*}
\inf_{(x,z,t)\in D}-u_{z}(x,z,t)>0
\end{equation*}
where $D=\left\{(x,z,t)\in\mathbb{R}^{n}\times[T,\infty)||u(x,z,t)-\Phi_{c}(0)|\le\min(1-\Phi_{c}(0),\Phi_{c}(0))\right\}$. 
\end{cor}

\begin{lem}[Decay of x-derivatives]\label{lem10}
Let $u(x,z,t)$ be a solution of (\ref{MRea-Diff}). Then for any constant $R>0$, it holds that
\begin{equation*}
\lim_{t\to\infty}\sup_{x\in\mathbb{R}^{n-1},|z|\le R}|u_{x_{i}}(x,z,t)|=0,\quad\lim_{t\to\infty}\sup_{x\in\mathbb{R}^{n-1},|z|\le R}|u_{x_{i}x_{j}}(x,z,t)|=0,
\end{equation*}
for each $1\le i,j\le n-1$.
\end{lem}

\section{Level set of the solutions}
As we mention in the previous section, we only give statements of lemmas  and omit their proof.
\begin{lem}[Level set]\label{lem11}
Let $u(x,z,t)$ be a solution of ($\ref{MRea-Diff}$) and $T>0$ be as defined in Corollary 2. Then there exists a smooth bounded function $\Gamma(x,t)$ such that
\begin{equation*}
u(x,z,t)=\Phi_{c}(0)\quad\mbox{if and only if}\quad z=\Gamma(x,t),
\end{equation*}
for any $(x,t)\in\mathbb{R}^{n-1}\times[T,\infty)$. Furthermore the following estimates hold:
\begin{enumerate}
\item For each $1\le i,j \le n-1$,
\begin{equation*}
\lim_{t\to\infty}\sup_{x\in\mathbb{R}^{n-1}}|\Gamma_{x_{i}}(x,t)|=0,\quad\lim_{t\to\infty}\sup_{x\in\mathbb{R}^{n-1}}|\Gamma_{x_{i}x_{j}}(x,t)|=0,
\end{equation*}
\item There exists a constant $M>0$ such that, for each $1\le i,j,k \le n-1$,
\end{enumerate}
\begin{equation*}
\sup_{x\in\mathbb{R}^{n-1}}|\Gamma_{x_{i}x_{j}x_{k}}|\le M,\quad \mbox{for}\quad t\ge T.
\end{equation*}
\end{lem}

\begin{lem}\label{lem12}
Let $u(x,z,t)$ be a solution of (\ref{MRea-Diff}) and let $\Gamma(x,t)$ be as defined in Lemma \ref{lem11}. Then it holds that
\begin{equation*}
\lim_{t\to\infty}\sup_{(x,z)\in\mathbb{R}^{n}}\left|u(x,z,t)-\Phi_{c}\left(x-\Gamma(x,t)\right)\right|=0
\end{equation*}
\end{lem}

\section{Construction of supersolutions and subsolutions}
In this section, we construct supersolutions and subsolutions to prove Theorem \ref{thm1} (ii).
For this purpose, let $V(x,t)$ be 
\begin{equation*}
\left\{
\begin{aligned}
&V_{t}=\Delta_{x}V+\frac{c^{*}}{2}|\nabla_{x} V|\\
&V(x,0)=V_{0}(x)
\end{aligned}
\right.
\end{equation*}

Then, the following holds

\begin{lem}\label{lem13}
For any $M>0,\varepsilon\in(0,1]$, there exists $\delta>0$ and smooth functions $p(t),q(t)$ such that
\begin{equation*}
p(0)>0,\; q(0)=0,\; 0\le p(t), q(t)\le\varepsilon\;(t\ge0)
\end{equation*}
and, if $\|V_{0}\|\le M, \|\nabla_{x}V_{0}\|\le\delta$, 
\begin{equation*}
u^{+}(x,z,t)\coloneqq\Phi_{c^{*}}\left(\frac{z-V(x,t)}{\sqrt{1+|\nabla_{x}V|^{2}}}-q(t)\right)+p(t)\psi(z)
\end{equation*}
becomes a supersolution.
\end{lem}

\begin{proof}
$\mathrm{L}[u]=I+J$, where $I,J$ are
\begin{equation*}
I\coloneqq(I_{0}-I_{2})\Phi_{c^{*}}^{\prime}+(I_{1}-3I_{3})\eta\Phi_{c^{*}}^{\prime}-2I_{2}\eta\Phi^{\prime\prime}_{c^*}-I_{3}\eta^{2}\Phi_{c^{*}}^{\prime\prime}
\end{equation*}
and $I_{0},I_{1},I_{2},I_{3}$ are 
\begin{align*}
&I_{0}\coloneqq-\frac{V_{t}+c^{*}}{\sqrt{1+|\nabla_{x}V|^{2}}}+\mathrm{div}\left(\frac{\nabla_{x}V}{\sqrt{1+|\nabla_{x}V|^{2}}}\right)+c^{*}\\
&I_{1}\coloneqq-\sum_{i=1}^{n-1}\frac{V_{x_{i}}V_{x_{i}t}}{1+|\nabla_{x}V|^{2}}+\sum_{i,j=1}^{n-1}\frac{V_{x_{i}x_{j}}^{2}+V_{x_{j}}V_{x_{i}x_{i}x_{j}}}{1+|\nabla_{x}V|^{2}}\\
&I_{2}\coloneqq\sum_{i,j=1}^{n-1}\frac{V_{x_{i}}V_{x_{j}}V_{x_{i}x_{j}}}{(1+|\nabla_{x}V|^{2}|)^{\frac{3}{2}}}\\
&I_{3}\coloneqq\sum_{i=1}^{n-1}\left(\frac{\sum_{j=1}^{n-1}V_{x_{j}}V_{x_{i}x_{j}}}{1+|\nabla_{x}V|^{2}}\right)^{2}.
\end{align*}
$I_{0},J$ is calculated as follows:
\begin{align*}
I_{0}=&\frac{1}{\sqrt{1+|\nabla_{x}V|^{2}}}\left(-V_{t}+\Delta_{x}V+\frac{c^{*}}{2}|\nabla_{x}V|\right)\\
&-\frac{c^{*}}{2\sqrt{1+|\nabla_{x}V|^{2}}(\sqrt{1+|\nabla_{x}V|}+1)^{2}}-\sum_{i,j=1}^{n-1}\frac{V_{x_{i}}V_{x_{j}}V_{x_{i}x_{j}}}{(1+|\nabla_{x}V|^{2})^{\frac{3}{2}}}\\
=&-\frac{c^{*}}{2\sqrt{1+|\nabla_{x}V|^{2}}(\sqrt{1+|\nabla_{x}V|}+1)^{2}}-\sum_{i,j=1}^{n-1}\frac{V_{x_{i}}V_{x_{j}}V_{x_{i}x_{j}}}{(1+|\nabla_{x}V|^{2})^{\frac{3}{2}}}
\end{align*}
\begin{equation*}
J\coloneqq\left(-\Phi_{c^{*}}^{\prime}\frac{q^{\prime}(t)}{p(t)\psi(z)}+\frac{p^{\prime}(t)}{p(t)}-\frac{\psi^{\prime\prime}(z)+c^{*}\psi^{\prime}(z)}{\psi(z)}-\int_{0}^{1}f^{\prime}(\Phi_{c^{*}}+\tau p(t)\psi(z))\mathrm{d}\tau\right)\cdot p(t)\psi(z)
\end{equation*}
Thus, there exists $S>0,C_{2}\ge1$ such that, for any $C_{1}>0$, there exists $\delta>0$ such that, if $\|\nabla_{x}V_{0}\|_{W^{1,\infty}}\le\delta$, then
\begin{gather*}
|I|\le P(t)(|\Phi_{c^{*}}^{\prime}|+|\eta||\Phi_{c^{*}}^{\prime}|+|\eta|^{2}|\Phi_{c^{*}}^{\prime\prime}|)\le SP(t)\\
P(t)\coloneqq\min\{C_{2}t^{-2},C_{1}\}.
\end{gather*}
By these, we can prove that $u^{+}$ is a supersolution by showing $J\ge |I|$\\
We define $K\in(0,1],L>0,\delta>0,R>0,K_{1}>0,K_{2}\in(0,1], R_{1}>0$ such that
\begin{gather*}
0<K<|\lambda^{2}+c^{*}\lambda+f^{\prime}(0)|\\
f^{\prime}(s)\ge2K>0\;(s\in[1-2\varepsilon,1+\varepsilon])\\
-K-\lambda^{2}-c^{*}\lambda-f^{\prime}(0)>L>0\\
\frac{\lambda_{-}(c^{*})}{\sqrt{1+\delta^{2}}}<\lambda\\
\left|\frac{\psi^{\prime\prime}+c^{*}\psi^{\prime}}{\psi}\right|\le K_{1}\\.
\end{gather*}
For any $z>R$, 
\begin{gather*}
e^{\lambda z}<\frac{1}{2}\\
-K-\lambda^{2}-c^{*}\lambda-\int_{0}^{1}f^{\prime}(\Phi+p\psi\tau)\mathrm{d}\tau\ge L>0\\
|\Phi^{\prime}|+|\eta||\Phi^{\prime}|+|\eta|^{2}|\Phi^{\prime\prime}|\le Le^{\lambda z}.
\end{gather*}
For any $z<-R$,
\begin{gather*}
\Phi\in[1-\varepsilon,1]\\
|\Phi^{\prime}|+|\eta||\Phi^{\prime}|+|\eta|^{2}|\Phi^{\prime\prime}|\le1.
\end{gather*}
For any $-R\le z\le R$,
\begin{gather*}
\psi\ge K_{2}>0\\
-R_{1}\le \eta-q(t)\le R_{1}.
\end{gather*}
Based on these, we define $C_{1}>0,C_{0}>1$ as follows:
\begin{equation*}
C_{1}\coloneqq\frac{K^{2}\varepsilon^{2}}{16C_{2}C_{0}^{2}}, C_{0}\coloneqq \max\left\{1,\frac{SK/K_{2}+K_{1}+K+\|f^{\prime}\|_{L^{\infty}[0,1]}}{\min_{|z|\le R_{1}|\Phi^{\prime}(z)|}}\right\}
\end{equation*}
Then, we take $p,q\in C^{\infty}[0,\infty)$ such that
\begin{equation*}
P(t)\le Kp(t)\le 2P(t),\hfill K|p^{\prime}(t)|\le2|P^{\prime}(t)|,\hfill q(t)=C_{0}\int_{0}^{t}p(s)\mathrm{ds}.
\end{equation*}
For these, the following holds
\begin{equation*}
p(0)\ge \frac{K\varepsilon}{16C_{2}C_{0}}>0,\hfill 0<p(t)\le\frac{K\varepsilon^{2}}{8C_{2}C_{0}^{2}},\hfill0\le q(t)\le C_{0}\int_{0}^{\infty}p(s)\mathrm{ds}\le\varepsilon.
\end{equation*}
Now, we prove that $u^{+}$ is a supersolution.\\
For any $z\le-R$,
\begin{align*}
J&=\left(-\Phi^{\prime}\frac{q^{\prime}}{p\psi}+\frac{p^{\prime}}{p}-\frac{\psi^{\prime\prime}+c^{*}\psi^{\prime}}{\psi}-\int_{0}^{1}f^{\prime}(\Phi+p\psi\tau)\mathrm{d}\tau\right)p\psi\\
&\ge\left(\frac{p^{\prime}}{p}-\int_{0}^{1}f^{\prime}(\Phi+p\psi\tau)\mathrm{d}\tau\right)p\\
&\ge Kp\ge|I|
\end{align*}
, where we use the following:
\begin{equation*}
\sup_{t\ge0}\frac{|p^{\prime}(t)|}{p(t)}\le\sup_{t\ge0}\frac{|P^{\prime}(t)|}{Kp(t)}\le\sup_{t\ge0}\frac{2|P^{\prime}(t)|}{P(t)}=\frac{K\varepsilon}{C_{2}C_{0}}\le K.
\end{equation*}
For any $z\ge R$,
\begin{align*}
J&\ge\left(-K-\lambda^{2}-c^{*}\lambda-\int_{0}^{1}f^{\prime}(\Phi+p\psi\tau)\mathrm{d}\tau\right)p\psi\\
&\ge Lpe^{\lambda z}\ge\frac{L}{K}P(t)e^{\lambda z}\ge LP(t)e^{\lambda z}\ge |I|
\end{align*}
For any $-R\le z\le R$,
\begin{align*}
J&\ge\left(-\Phi^{\prime}\frac{C_{0}}{\psi}+\frac{p^{\prime}}{p}-K_{1}-\|f^{\prime}\|_{L^{\infty}[0,1]}\right)p\psi\\
&\ge\left(-\Phi^{\prime}C_{0}-K-K_{1}-\|f^{\prime}\|_{L^{\infty}[0,1]}\right)K_{2}p\\
&\ge SKp(t)\ge SP(t)\ge |I|.
\end{align*}
Therefore, $u^{+}$ is a supersolution.
\end{proof}

subsolutions are as follows:

\begin{lem}\label{lem14}
For any $M>0,\varepsilon\in(0,1]$, there exists $\delta>0$ and smooth functions $p(t),q(t)$ such that
\begin{equation*}
p(0)>0,\; q(0)=0,\; 0\le p(t), q(t)\le\varepsilon\;(t\ge0)
\end{equation*}
and, if $\|V_{0}\|\le M, \|\nabla_{x}V_{0}\|\le\delta$, 
\begin{equation*}
u^{-}(x,z,t)\coloneqq\Phi_{c^{*}}\left(\frac{z-V(x,t)}{\sqrt{1+|\nabla_{x}V|^{2}}}+q(t)\right)-p(t)\psi(z)
\end{equation*}
becomes a supersolution.
\end{lem}

\section{Proof of Main theorem}
 In this section, we complete the proof of the main theorem by proving the statement Theorem \ref{thm1} (ii).
 
\begin{lem}[Approximation of $\Gamma(x,t)$]\label{lem15}
Let $u(x,z,t)$ be a solution of (\ref{MRea-Diff}) and let $\Gamma(x,t)$ be as defined in Lemma \ref{lem11}. Then for any $\varepsilon>0$, there exists a constant $\tau_{\varepsilon}>0$ such that the function $V(x,t)$ defined by
\begin{equation*}
\left\{
\begin{aligned}
&V_{t}=\Delta_{x}V+\frac{c^{*}}{2}|\nabla_{x} V|,\quad x\in\mathbb{R}^{n-1},t>0\\
&V(x,0)=\Gamma(x,\tau_{\varepsilon}).\qquad\hspace{9.4pt} x\in\mathbb{R}^{n-1}.
\end{aligned}
\right.
\end{equation*}
satisfies
\begin{equation*}
\sup_{x\in\mathbb{R}^{n-1}}|\Gamma(x,t)-V(x,t-\tau_{\varepsilon})|\le\varepsilon,\quad t\ge\tau_{\varepsilon}.
\end{equation*}
\end{lem}

\begin{proof}
 First, we verify upper bound. By Lemma \ref{lem12} and \ref{lem13}, we can take $T>0, M>0$ and $K>0$ such that, for $D\coloneqq\left\{(x,z,t)\in\mathbb{R}^{n}\times[T,\infty)||u(x,z,t)-\Phi_{c}(0)|\le\min(1-\Phi_{c}(0),\Phi_{c}(0))\right\}$
\begin{equation*}
\sup_{t\ge T}\|\Gamma(\cdot,t)\|_{W^{3,\infty}}\le M,\quad \inf_{(x,z,t)\in D}-u_{z}(x,z,t)\ge K.
\end{equation*}
For the constants $M$ and $\hat{\varepsilon}\coloneq1/(\|\Phi_{c^{*}}^{\prime}\|_{L^{\infty}}+1)\cdot\min\{K\varepsilon,\min(1-\Phi_{c^{*}}(0),\Phi_{c^{*}}(0))\}$, we choose a constant $\delta>0$ and functions $p(t),q(t)$ satisfying
\begin{equation*}
p(0)>0,\quad q(0)=0,\quad 0\le p(t),q(t)\le\hat{\varepsilon}\quad for \quad t\ge0.
\end{equation*}
From proof of Lemma \ref{lem3} and (\ref{asym}), we can take $z_{0}\in\mathbb{R}$ such that, for some larger $T>0$,  
\begin{equation*}
u(x,z,t)\le p(0)e^{\lambda z}\quad((x,z,t)\in\mathbb{R}^{n-1}\times[z_{0},\infty)\times[T,\infty)).
\end{equation*}
Taking $\tau_{\varepsilon}\ge T$ larger if necessary, by Lemma \ref{lem11} and \ref{lem12}, the following holds
\begin{equation*}
u(x,z,t)\le\Phi_{c^{*}}(z-\Gamma(x,\tau_{\varepsilon}))+\frac{p(0)e^{\lambda z_{0}}}{2}\le\Phi_{c^{*}}\left(\frac{z-\Gamma(x,\tau_{\varepsilon})}{\sqrt{1+|\nabla_{x}\Gamma|^{2}}}\right)+p(0)e^{\lambda z_{0}}
\end{equation*}
For any $z\ge z_{0}$, 
\begin{equation*}
u(x,z,\tau_{\varepsilon})\le p(0)e^{\lambda z}\le u^{+}(x,z,0).
\end{equation*}
For any $z<z_{0}$,
\begin{equation*}
u(x,z,\tau_{\varepsilon})\le \Phi_{c^{*}}\left(\frac{z-\Gamma(x,\tau_{\varepsilon})}{\sqrt{1+|\nabla_{x}\Gamma|^{2}}}\right)+p(0)e^{\lambda z_{0}}\le\Phi_{c^{*}}\left(\frac{z-\Gamma(x,\tau_{\varepsilon})}{\sqrt{1+|\nabla_{x}\Gamma|^{2}}}\right)+p(0)\psi(z).
\end{equation*}
So, by comparison principle, $u(x,z,t)\le u^{+}(x,z,t)$ for $t\ge\tau_{\varepsilon}$. Therefore, we have the following inequality
\begin{align*}
u(x,V(x,t-\tau_{\varepsilon}),t)-\Phi_{c^{*}}(0)&\le u^{+}(x,V(x,t-\tau_{\varepsilon}),t)-\Phi_{c^{*}}(0)\\
                                         &=\Phi_{c^{*}}(-q(t-\tau_{\varepsilon}))-\Phi_{c^{*}}(0)+p(t-\tau_{\varepsilon})\\
                                         &\le(\|\Phi_{c^{*}}^{\prime}\|_{L^{\infty}}+1)\hat{\varepsilon}\\
                                         &=\min\{K\varepsilon,\min\{1-\Phi_{c^{*}}(0),\Phi_{c^{*}}(0)\}\}.
\end{align*}
Thus, we have, for $\Gamma(x,t)\ge V(x,t-\tau_{\varepsilon})$,
\begin{align*}
K\varepsilon&\ge u(x,V(x,t-\tau_{\varepsilon},t),t)-u(x,\Gamma(x,t),t)\\
                  &\ge (\inf_{u\in[0,\min\{1-\Phi_{c^{*}}(0),\Phi_{c^{*}}(0)\}],t\ge\tau_{\varepsilon}}-u_{z})\cdot(\Gamma(x,t)-V(x,t-\tau_{\varepsilon}))\\
                  &\ge K(\Gamma(x,t)-V(x,t-\tau_{\varepsilon})).
\end{align*}
This implies $\Gamma(x,t)\le V(x,t-\tau_{\varepsilon})+\varepsilon$ for $t\ge\tau_{\varepsilon}$. Thus, we have the upper estimate. The lower estimate is followed from Lemma \ref{lem14} in a similar way.
\end{proof}

Based on these, we give a proof of Theorem 1.

\begin{proof}[Proof of Theorem \ref{thm1}]
The statements (i) and (ii) of Theorem \ref{thm1} are derived directly from Lemmas \ref{lem11} and \ref{lem12}, respectively. Thus, we only show Theorem \ref{thm1} (ii). By Lemma \ref{lem15}, the large time behavior of the level-surface $\Gamma(x,t)$ of the solution $u(x,z,t)$ is approximated by the solution $V(x,t)$ of the equation
\begin{equation*}
V_{t}=\Delta_{x}V+\frac{c^{*}}{2}|\nabla_{x}V|^{2},\quad x\in\mathbb{R}^{n-1},\quad t>0.
\end{equation*}
This means that the level-surface $\gamma(x,t)=\Gamma(x,t)+ct$ of the solution $u(x,z,t)$ of (\ref{MRea-Diff}) can be approximated by the solution $\hat{V}(x,t)$ of the equation
\begin{equation*}
\hat{V}_{t}=\Delta_{x}\hat{V}+\frac{c^{*}}{2}|\nabla_{x}\hat{V}|^{2}+c^*,\quad x\in\mathbb{R}^{n-1},\quad t>0.
\end{equation*}
Thus Theorem \ref{thm1} (ii) follows from Lemma \ref{lem15}. This completes the proof of Theorem \ref{thm1}.
\end{proof}

\section*{Acknowledgements}
The author would like to thank Toru Kan for his invaluable guidance and detailed feedback throughout this work. 
The author also thanks the anonymous referees for their careful reading and constructive comments which helped improve the manuscript

\bibliographystyle{abbrv}
\bibliography{English_draft.bib}
\end{document}